\documentclass[twoside]{aiml22}

\usepackage{aiml22macro}

\usepackage{graphicx}
\usepackage{amsmath}
\usepackage{amssymb}
\usepackage{amsfonts}
\usepackage{bm}
\usepackage{tikz-cd}

\newcommand{\RO}{\mathsf{RO}}
\newcommand{\ROU}{\mathcal{RO}}
\newcommand{\ORO}{\mathsf{O}\ROU}
\newcommand{\inv}{^{-1}}
\newcommand{\reg}{^{\bot\bot}}
\newcommand{\sset}{\subseteq}
\newcommand{\rset}{\supseteq}
\newcommand{\me}{\wedge}
\newcommand{\bigme}{\bigwedge}
\newcommand{\jo}{\vee}
\newcommand{\bigjo}{\bigvee}
\newcommand{\orclo}{\mathord{\downarrow}}

\newcommand{\cat}{\mathbf}
\newcommand{\concord}{\mathord{\mathrel{\rotatebox[origin=c]{90}{$\twoheadrightarrow$}}}}
\newcommand{\orint}{\mathord{\mathrel{\rotatebox[origin=c]{90}{$\rightarrowtail$}}}}

\def\lastname{Massas}

\begin{document}

\begin{frontmatter}
  \title{Choice-Free de Vries Duality}
  \author{Guillaume Massas}\footnote{gmassas@berkeley.edu}
  \address{University of California, Berkeley}

  \begin{abstract}
  De Vries duality generalizes Stone duality between Boolean algebras and Stone spaces to a duality between de Vries algebras (complete Boolean algebras equipped with a subordination relation satisfying some axioms) and compact Hausdorff spaces. This duality allows for an algebraic approach to region-based theories of space that differs from point-free topology. Building on the recent choice-free version of Stone duality developed by Bezhanishvili and Holliday, this paper establishes a choice-free duality between de Vries algebras and a category of de Vries spaces. We also investigate connections with the Vietoris functor on the category of compact Hausdorff spaces and with the category of compact regular frames in point-free topology, and we provide an alternative, choice-free topological semantics for the Symmetric Strict Implication Calculus of Bezhanishvili et al.
  \end{abstract}

  \begin{keyword}
  Duality theory, de Vries algebras, point-free topology.
  \end{keyword}
 \end{frontmatter}

\section{Introduction}

Stone's \cite{stone1936theory} representation of Boolean algebras as clopen sets of compact, Hausdorff and zero-dimensional topological spaces has had a profound influence on the study of interactions between logic, algebra and topology. The realization that some properties of topological spaces could be retrieved by considering the algebraic properties of their lattices of open sets led to the development of \textit{point-free} topology \cite{Johnstone,johnstone1983point,picado2011frames}, in which open sets are taken as basic elements of a \textit{frame} rather than defined as sets of points. Stone's representation theorem, and therefore Stone duality, relies on the Boolean Prime Ideal Theorem (BPI), a fragment of the Axiom of Choice. By contrast,  the point-free approach has a more constructive flavor: even in the absence of the Axiom of Choice, the open set functor $\Omega$ mapping a topological space to its lattice of open sets has an adjoint functor $pt$, mapping a frame to its set of ``points'' endowed with a natural Stone-like topology. But the restriction of this adjunction to Stone spaces and compact zero-dimensional frames is only a duality under (BPI).  In \cite{hol19}, a choice-free duality between Boolean algebras and a category of $UV$-spaces has been developed. It is based on the simple but powerful idea that the appeal to (BPI) could be eliminated by working with a partially-ordered set of filters rather than a set of ultrafilters and by viewing these filters as partial approximations of a classical point. This approach has strong ties to both possibility semantics in modal logic \cite{holposs,Hol16,holliday2019complete} and the Vietoris functor on Stone spaces \cite{vietoris1922bereiche} and provides a \textit{semi-pointfree} approach, i.e., both spatial \textit{and} choice-free, to the representation of algebraic objects in \textit{semi-constructive mathematics}, i.e., mathematics carried out in $ZF+DC$ \cite{massas2022semi,schechter1996handbook}.

In \cite{devries1962compact}, de Vries generalized Stone duality to a duality between de Vries algebras (complete Boolean algebras equipped with a \textit{subordination} relation) and compact Hausdorff spaces. Just like Stone, de Vries used (BPI) in his representation of complete compingent algebras as the regular open sets of a compact Hausdorff space. On the point-free side, Isbell \cite{isbell1972atomless} showed that the $\Omega$-$pt$ adjunction restricts to a duality between compact Hausdorff spaces and compact regular frames, also under the assumption of (BPI). This leaves open the question of whether a choice-free duality between these algebraic categories and a category of topological spaces can be defined. 

In this paper, we show that this is indeed possible by generalizing the approach of \cite{hol19}. Just like Bezhanishvili and Holliday, we work with a poset of filters rather than with a set of maximal filters, and we define our dual spaces both in terms of their topological properties and in terms of order-theoretic aspects of the induced specialization order. We also show how the spaces we define naturally relate to the Vietoris functor on compact Hausdorff spaces and compact regular frames. We take this as evidence of the naturality and fruitfulness of this semi-pointfree approach, in which the basic ``points'' of our spaces are the closed sets of the standard, non-constructive duality.

The paper is organized as follows. After reviewing some background on de Vries algebras, compact regular frames and $UV$-spaces (Section \ref{back}), we provide a choice-free representation of any de Vries algebra as the regular open sets of some topological space (Section \ref{rep}). In Section \ref{space}, we characterize the choice-free duals of de Vries algebras, which we call $dV$-spaces. Section \ref{morph} deals with morphisms and ends with our main result, a choice-free dual equivalence between the category of de Vries algebras and the category of $dV$-spaces. In Section \ref{hyp}, we connect our duality to point-free topology and provide an alternative characterization of $dV$-spaces via the Vietoris functor on compact regular frames. Finally, Section \ref{app} lists two applications of this duality: a choice-free analogue of Tychonoff's Theorem for compact Hausdorff spaces and a choice-free topological semantics for the system $\mathsf{S^2IC}$ introduced in \cite{bezhanishvili2019strict}.

\section{Background}\label{back}
In this section, we briefly recall the de Vries and Isbell dualities for compact Hausdorff spaces as well as the choice-free Stone duality between Boolean algebras and $UV$-spaces presented in \cite{hol19}. We start by fixing some notation that we will use throughout. Let $L$ be a complete lattice and $(X, \tau)$ be a topological space. 

\begin{enumerate}
    \item When no confusion arises, we write $\leq$ to designate the order on $L$. We designate the meet and join operations on $L$ by $\me$ and $\jo$ respectively, and, whenever $L$ is pseudo-complemented, we use $\neg$ for the pseudo-complement operation.
    \item We will designate (a subset of) the set of all maximal filters on $L$ by $X_L$ and (a subset of) the set of all filters on $L$ by $S_L$.
    \item By a Stone-like topology on a set $Y$ of filters of $L$, we mean the topology generated by the sets of the form $\widehat{a} = \{F \in Y \mid a \in F\}$, and we will usually designate such a topology by $\sigma$.
    \item For any $U \sset X$, we write $-U$ for $X \setminus U$, $\overline{U}$ for the closure of $U$ and $U^\bot$ for $-\overline{U}$. We write $\mathsf{CO}(X)$ for the set of compact open subsets of $X$ and $\RO(X)$ for the Boolean algebra of regular open subsets of $X$, i.e., subsets $U$ such that $U\reg = U$.
    \item The \textit{specialization preorder} on $(X, \tau)$ is represented by the symbol $\leq$ when no confusion arises, and it is defined as $x \leq y$ iff $x \in U$ implies $y \in U$ for every $U \in \tau$.
    \item The \textit{up-set topology} on $X$ is the topology generated by the set of all upward closed subsets in the specialization preorder. Given $ U \sset X$, we let $\orint U$ be the interior of $U$ in the up-set topology, and $\orclo U$ the closure of $U$. We write $\ROU(X)$ for the Boolean algebra of order-regular open subsets of $X$, i.e., subsets $U$ such that $\orint \orclo U = U$, and $\mathsf{CO}\ROU(X)$ for $\mathsf{CO}(X)\cap\ROU(X)$.
\end{enumerate}
\subsection{De Vries Algebras} \label{devriesdu}

De Vries algebras were introduced in \cite{devries1962compact} as an algebraic dual to compact Hausdorff spaces. 

\begin{definition}
A \textit{compingent algebra} is a pair $(B, \prec)$ such that $B$ is a Boolean algebra with induced order $\leq$, and $\prec$ is a relation on $B \times B$ satisfying the following set of axioms:

\begin{description}
    \item[(A1)] $1 \prec 1$;
    \item[(A2)] $a \prec b$ implies $a \leq b$;
    \item[(A3)] $a \leq b \prec c \leq d$ implies $a \prec d$;
    \item[(A4)] $a \prec b$ and $a \prec c$ together imply $a \prec b \me c $;
    \item[(A5)] $a \prec b$ implies $\neg b \prec \neg a$;
    \item[(A6)] $a \prec c$ implies that there is $b \in B$ such that $a \prec b \prec c$;
    \item[(A7)] $a \neq 0$ implies that there is $b \neq 0 \in B$ such that $b \prec a$.
\end{description}

A \textit{de Vries algebra} is a compingent algebra $V = (B, \prec)$ such that $B$ is a complete Boolean algebra. It is \textit{zero-dimensional} if for any $a \prec b \in V$ there is $c \in V$ such that $a \prec c \prec c \prec b$.
\end{definition}

Compingent algebras constitute a specific kind of \textit{contact algebras}, Boolean algebras equipped with a binary relation of subordination satisfying (A1)-(A5). One motivation for contact algebras is to develop a region-based theory of space \cite{dimov2006contact,lando2019calculus}, according to which regions of space form a Boolean algebra and a region $a$ is subordinated to a region $b$ precisely if $b$ completely surrounds $a$. For more on contact  and subordination algebras, we refer the reader to \cite{bezhanishvili2019strict,bezhanishvili2017irreducible,dimov2017generalization,fedorchuk1973boolean}.

\begin{definition}
Let $V = (B, \prec)$ be a de Vries algebra. For any filter $F$ on $B$, let $\concord F = \{a \in F \mid \exists b \in F : b \prec a \}$. A \textit{concordant filter} on $V$ is a filter $F$ such that $\concord F = F$. An \textit{end} is a maximal concordant filter. 
\end{definition}

The dual space of a de Vries algebra $V$ is obtained by taking the set $X_V$ of all ends of $V$ and endowing it with the Stone-like topology $\sigma$ generated by all sets of the form $\{ p \in X_V \mid a \in p\}$ for some $p \in V$. Conversely, the dual de Vries algebra of a compact Hausdorff space $(X,\tau)$ is the complete Boolean algebra $\RO(X)$ of regular open sets, with the subordination relation $\sqsubset$ given by $U \sqsubset V$ iff $\overline{U} \sset V$.

\begin{theorem}[\cite{devries1962compact}, Thm.~I.4.3-5]
For any de Vries algebra $V = (B, \prec)$, $(X_V, \sigma)$ is compact Hausdorff, and $(B, \prec)$ is isomorphic to $(\RO(X_V),\sqsubset)$. Conversely, for any compact Hausdorff space $(X, \tau)$, $(\RO(X),\sqsubset)$ is a de Vries algebra, and $(X,\tau)$ is homeomorphic to $(X_{(\RO(X),\sqsubset)},\sigma)$. 
\end{theorem}

We now introduce the relevant notion of morphism between de Vries algebras.

\begin{definition}
Let $V_1 = (B_1,\prec_1)$ and $V_2 = (B_2,\prec_2)$ be de Vries algebras. A de Vries morphism from $V_1$ to $V_2$ is a function $h : B_1 \to B_2$ satisfying the following set of conditions:
\begin{description}
    \item[(V1)] $h(0) = 0$;
    \item[(V2)] $h(a \me b) = h(a) \me h(b)$;
    \item[(V3)] $a \prec_1 b$ implies $\neg h(\neg a) \prec_2 h(b)$;
    \item[(V4)] $h(a) = \bigjo\{h(b) \mid b \prec_1 a\}$.
\end{description}
Given two de Vries morphisms $h: V_1 \to V_2$ and $k:V_2 \to V_3$, their composition $k\star h: V_1 \to V_3$ is defined as the map $a \mapsto \bigjo\{kh(b) : b \prec_1 a\}$. 
\end{definition}

One easily verifies that de Vries morphisms preserve both the order $\leq$ and the subordination relation $\prec$. Given a de Vries morphism $h : V_1 \to V_2$, the map $h^* : X_{V_2} \to X_{V_1}$ given by $h^*(p) = \concord h\inv[p]$ for any end $p$ on $V_2$ is a continuous function. Conversely, for any continuous function $f: (X_1,\tau_1) \to (X_2,\tau_2)$, the map $f_*:\RO(X) \to \RO(Y)$ given by $f_*(U) = (f\inv[U])\reg$ for any regular open set $U$ is a de Vries morphism. This allowed de Vries to obtain the following:

\begin{theorem} \label{devries}
The category $\cat{deV}$ of de Vries algebras and de Vries morphisms between them is dually equivalent to the category $\cat{KHaus}$ of compact Hausdorff spaces and continuous maps between them. 
\end{theorem}

\subsection{Compact Regular Frames}
Recall that a \textit{frame} is a complete lattice $L$ that satisfies the join-infinite distributive law, i.e., is such that $a \me \bigjo B = \bigjo \{a \me b \mid b \in B\}$ for any $a \in L$ and $B \sset L$. Frames are the central object of study of point-free topology, for which \cite{Johnstone,johnstone1983point,picado2011frames} are standard introductions. A frame $L$ is \textit{compact} if $1_L = \bigjo B$ for some $B \sset L$ implies that $1_L = \bigjo B'$ for some finite $B' \sset B$. A morphism between frames is a map preserving finite meets and arbitrary joins.

\begin{definition}
Let $L$ be a frame and $a,b \in L$. Then $a$ is said to be \textit{rather below} $b$ \cite[Def.~V.5.2]{picado2011frames}, denoted $a \prec b$, if $b \jo \neg a = 1_L$. A \textit{compact regular frame} is a compact frame $L$ such that for any $a \in L$, $a = \bigjo\{b \in L \mid b \prec a\}$.
\end{definition}

Given any topological space $(X, \tau)$, one can define its frame of open sets $\Omega(X)$. Conversely, given a frame $L$, one can define a Stone-like topology on the set of completely prime filters $pt(L)$. These constructions give rise to adjoint functors $\Omega$ and $pt$ between the categories $\cat{Frm}$ of frames and frame morphisms and $\cat{Top}$ of topological spaces and continuous functions. Assuming (BPI), Isbell \cite{isbell1972atomless} showed that this adjunction restricts to a duality in the specific case of compact regular frames:

\begin{theorem} \label{isbell}
The category $\cat{KRFrm}$ of compact regular frames is dually equivalent to  $\cat{KHaus}$.
\end{theorem}

As an immediate consequence of Theorems \ref{devries} and \ref{isbell}, the categories $\cat{deV}$ and $\cat{KRFrm}$ are equivalent. This equivalence has also been given a direct description in \cite{bezhanishvili2012vries}, which has the advantage of being choice-free. Given a frame $L$, an element $a \in L$ is \textit{regular} if $\neg\neg a = a$. The \textit{Booleanization} of $L$ \cite{banaschewski1996booleanization}, denoted $B(L)$, is the subframe of all the regular elements of $L$. It is straightforward to verify that if $L$ is a compact regular frame, $B(L)$ equipped with the rather below relation $\prec$ is a de Vries algebra. In order to go from de Vries algebras to frames, we need the following definition:

\begin{definition}
Let $V = (B, \prec)$ be a de Vries algebra. An ideal on $B$ is \textit{round} if for any $a \in I$, there is $b \in I$ such that $a \prec b$.
\end{definition}

It is immediate to see that a proper ideal $I$ on a de Vries algebra $V$ is round if and only if its dual filter $I^\delta = \{\neg a \mid a \in I\}$ is concordant. Given a de Vries algebra $V$, its set of round ideals ordered by inclusion forms a compact regular frame $\mathfrak{R}(V)$. The equivalence between $\cat{KRFrm}$ and $\cat{deV}$ is then given by the following result:

\begin{theorem} \label{gur}
Any compact regular frame $L$ is isomorphic to $\mathfrak{R}(B(L))$. Conversely, any de Vries algebra $V$ is isomorphic to $B(\mathfrak{R}(V))$, and the maps $B$ and $\mathfrak{R}$ lift to an equivalence between $\cat{KRFrm}$ and $\cat{deV}$.
\end{theorem}

\subsection{$UV$-spaces}\label{chfstone}
We conclude this section by recalling the choice-free version of Stone duality presented in \cite{hol19}. 

\begin{definition}
A topological space $(X,\tau)$ is a \textit{$UV$-space} if it satisfies the following conditions:
\begin{enumerate}
    \item $(X,\tau)$ is compact and $T_0$;
    \item $\mathsf{CO}\ROU(X)$ is closed under $\cap$ and $-\orclo$ and forms a basis for $\tau$;
    \item Any filter on $\mathsf{CO}\ROU(X)$ is $\mathsf{CO}\ROU(x) = \{U \in \mathsf{CO}\ROU(X) \mid x \in U\}$ for some $x \in X$.
\end{enumerate}
 
\end{definition}

Given a Boolean algebra $B$, one considers the set $S_B$ of all filters on $B$, equipped with the usual Stone-like topology $\sigma$. It can then be showed without appealing to (BPI) that $UV$-spaces are the duals of Boolean algebras:

\begin{theorem}[\cite{hol19}, Thm.~5.4]
For any Boolean algebra $B$, $(S_B, \sigma)$ is a $UV$-space, and $B$ is isomorphic to $\mathsf{CO}\ROU(S_B)$. Conversely, for any $UV$-space $(X,\tau)$, $\mathsf{CO}\ROU(X)$ is a Boolean algebra, and $(X,\tau)$ is homeomorphic to $(S_{\mathsf{CO}\ROU(X)},\sigma)$.
\end{theorem}

Moving on to morphisms, recall that a spectral map between two topological spaces is a map such that the preimage of any compact open set is compact open.

\begin{definition}
Given two $UV$-spaces $(X,\tau_1)$ and $(Y,\tau_2)$ with induced specialization orders $\leq_1$ and $\leq_2$, a $UV$-map from $(X,\tau_1)$ to $(Y,\tau_2)$ is a spectral map $f: X \to Y$ that is also a p-morphism with respect to $\leq_1$ and $\leq_2$, i.e., for any $x \in X$, $y \in Y$, if $f(x) \leq_2 y$, then there is $x'\geq_1 x$ such that $y = f(x')$.
\end{definition}

Any $UV$-map $f : (X,\tau_1) \to (Y,\tau_2)$ gives rise to a Boolean algebra homomorphism $f_* : \mathsf{CO}\ROU(Y) \to \mathsf{CO}\ROU(X)$ given by $f_*(U) = f\inv[U]$ for any $U \in \ROU(Y)$. Conversely, for any Boolean homomorphism $h: B_1 \to B_2$, the map $h^*: (S_{B_2},\sigma_2) \to (S_{B_1},\sigma_1)$ given by $h^*(F) = h\inv[F]$ for any filter $F$ on $B_2$ is a $UV$-map. This yields the following result, which, unlike Stone duality, does not rely on the Axiom of Choice: 

\begin{theorem}
The category $\cat{BA}$ of Boolean algebras and Boolean homomorphisms between them is dually equivalent to the category $\cat{UV}$ of $UV$-spaces and $UV$-maps between them. 
\end{theorem}

\section{A Choice-free Representation for de Vries Algebras}\label{rep}

In this section, we complete the first step of the duality by obtaining a choice-free representation of any de Vries algebra as the regular open sets of some topological space. Our approach combines the techniques of Sections \ref{devriesdu} and \ref{chfstone} in a natural way.

\begin{definition}
Let $V = (B, \prec)$ be a de Vries algebra. The \textit{dual filter space} of $V$ is the topological space $(S_V,\sigma)$, where:
\begin{itemize}
    \item $S_V$ is the set of all concordant filters on $V$;
    \item $\sigma$ is the Stone-like topology generated by $\{\widehat{a} = \{F \in S_V \mid a \in F\} \mid a \in V\}$.
\end{itemize}
\end{definition}

The following two lemmas will help us investigate the structure of the space of concordant filters on a de Vries algebra.

\begin{lemma} \label{filt}
Let $V = (B, \prec)$ be a de Vries algebra. Then:
\begin{enumerate}
    \item For any $a \neq 0$, $F = \{c \in V \mid a \prec c\}$ is a concordant filter.
    \item If $F$ and $G$ are concordant filters and $c \me d \neq 0$ for any $c \in F, d\in G$, then the set $H = \{c \me d \mid c \in F, d \in G\}$ is a concordant filter.
\end{enumerate}
\end{lemma}

\begin{proof}    For part (i), by (A3), $F$ is upward-closed, and by (A4), it is downward directed. To verify that $\concord F = F$, note that if $a \prec c$, then by (A6) there is $c'$ such that $a \prec c' \prec c$, so $c \in \concord F$.
    
    For part (ii), let $H = \{c \me d \mid c \in F, d \in G\}$. I claim that $H$ is a concordant filter. It is routine to verify that $H$ is a proper filter. To see that $\concord H = H$, take $c \in F$ and $d \in G$. Since $F$ and $G$ are concordant there are $c' \prec c$ in $F$ and $d' \prec d$ in $G$. Thus $c' \me d' \in H$ and $c' \me d' \prec c \me d$ by (A4), which means that $c \me d \in \concord H$. This shows that $H \sset \concord H$, and the converse is immediate from (A2).
\end{proof}

\begin{lemma}\label{reg}
Let $V = (B, \prec)$ be a de Vries algebra, $a \in V$ and $F$ a concordant filter on $V$. If $a \notin F$, then there is a concordant filter $G \rset F$ such that for any concordant filter $H \rset G$, $a \notin H$. 
\end{lemma}

\begin{proof}
Suppose $a \notin F$, and consider the set $G = \{c \me d \mid c \in F, \neg a \prec d\}$. I claim that $G$ is a concordant filter. If $c \me d = 0$ for some $c \in F$ and $d$ such that $\neg a \prec d$, then $c \leq \neg d \prec \neg \neg a = a$, which contradicts the assumption that $a \notin F$. Thus by Lemma \ref{filt} $G$ is a concordant filter. 

Now suppose $H$ is a concordant filter such that $H \rset G$. If $a \in H$, then there is $d \in H$ such that $d \prec a$. But this implies that $\neg a \prec \neg d$, so $\neg d \in G \sset H$, a contradiction.
\end{proof}

Given a de Vries algebra $V$ with dual space $(S_V,\sigma)$, we now show that the map $a \mapsto \widehat{a}$ is a Boolean embedding of $V$ into $\RO(S_V)$:
\begin{lemma}\label{basic}
Let $V = (B, \prec)$ be a de Vries algebra with dual filter space $(S_V, \sigma)$. Then for any $a, b \in V$:
\begin{enumerate}
    \item $\widehat{a} \cap \widehat{b} = \widehat{a \me b}$;
    \item The set $\{\widehat{a}\mid a \in V\}$ is a basis for $\sigma$, and the specialization order on $(S_V,\sigma)$ coincides with the inclusion order;
    \item $\widehat{a} \sset \widehat{b}$ iff  $a \leq b$;
    \item $\widehat{a}^\bot = \widehat{\neg a}$;
    \item $\orint \orclo \widehat{a} = \widehat{a} = \widehat{a}\reg$.
\end{enumerate}
\end{lemma}

\begin{proof}
Part (i) is a consequence of the fact that the elements of $S_V$ are filters, and part (ii) immediately follows from part (i). For part (iii), the right-to-left direction is immediate, and for the converse, since $B$ is a Boolean algebra it is enough to show that for any $a \neq 0$, there is some concordant filter $F$ such that $a \in F$. To see this, note that, by (A7), if $a \neq 0$ there is $b \neq 0$ such that $b \prec a$. Then $F = \{c \in V \mid b \prec c\}$ is a concordant filter by Lemma \ref{filt}, and $a \in F$.

For part (iv), since the set $\{\widehat{a}\mid a \in V\}$ is a basis for $\sigma$ by (ii), we have that for any $F \in S_V$, $F \in \overline{\widehat{a}}$ iff for any basic open $\widehat{b}$, $F \in \widehat{b}$ implies $\widehat{a} \cap \widehat{b} \neq \emptyset$. By (i) and (iii), this means that $F \in \overline{\widehat{a}}$ iff $b \me a \neq 0$ for all $b \in F$ iff $\neg a \notin F$ iff $F \notin \widehat{\neg a}$. Hence $\widehat{a}^\bot = \widehat{\neg a}$. 

Finally, for part (v), $\widehat{a} = \widehat{a}\reg$ follows directly from (iv). To show that $\orint \orclo \widehat{a} = \widehat{a}$, note that the right-to-left inclusion is immediate since $\widehat{a}$ is upward-closed. Since the specialization order on $(S_V,\sigma)$ coincides with the inclusion ordering, establishing the converse amounts to showing that for any $F \in S_V$, if $a \notin F$, then there is $G \rset F$ such that for all $H \rset G$, $a \notin H$. But this is precisely Lemma \ref{reg}.
\end{proof}

\begin{corollary} \label{cor1}
Let $V = (B, \prec)$ be a de Vries algebra with dual filter space $(S_V, \sigma)$. Then $B$ is isomorphic to $\RO(S_V)$.
\end{corollary}

\begin{proof}
Lemma \ref{basic} implies that the map $a \mapsto \widehat{a}$ is an injective Boolean homomorphism of $B$ into $\RO(S_V)$. Therefore it only remains to show that every regular open subset of $S_V$ is of the form $\widehat{a}$ for some $a \in V$. Let $U = \bigcup_{a \in A} \widehat{a}$ be a regular open set. Recall that $\bigjo A \in B$ since $B$ is a complete Boolean algebra. I claim that $\overline{U} = \overline{\widehat{\bigjo A}}$. Since $U$ is regular open, this will readily imply that $U = \widehat{\bigjo A}$. For the proof of the claim, recall that for any $F \in S_V$, $F \in \overline{\widehat{\bigjo A}}$ iff $\neg \bigjo A \notin F$. Similarly, $F \in \overline{U}$ iff for any $b \in F$ there is $a \in A$ such that $b \nleq \neg a$. But the latter condition is equivalent to $b \nleq \bigme \{ \neg a \mid a \in A\}$, which is in turn equivalent to $\neg \bigjo A \notin F$. Hence $F \in \overline{U}$ iff $F \in \overline{\widehat{\bigjo A}}$ for any $F \in S_V$, which means that $\overline{U} = \overline{\widehat{\bigjo A}}$. This completes the proof that $B$ is isomorphic to $\RO(S_V)$.
\end{proof}

We now turn to representing the subordination relation on a de Vries algebra. For any topological space $(X,\tau)$ and any $U,V \sset X$, let $U \ll V$ iff $\overline{U} \sset \orclo {V}$.

\begin{lemma} \label{prox}
Let $V = (B, \prec)$ be a de Vries algebra with dual filter space $(S_V, \sigma)$. For any $a, b \in V$, $a \prec b$ iff $\widehat{a} \ll \widehat{b}$.
\end{lemma}

\begin{proof}
For the first direction, suppose that $a \prec b$. Then if $F$ is a concordant filter such that $\neg a \notin F$, by Lemma \ref{filt} $G = \{c \me d \mid c \in F, a \prec d\}$ is a concordant filter extending $F$ and containing $b$. Now for any concordant filter $F$, $F \in \overline{\widehat{a}}$ iff $\neg a \notin F$. This shows that $\overline{\widehat{a}} \sset \orclo \widehat{b}$. 
Conversely, assume that $a \nprec b$. I claim that there is a concordant filter $F$ such that $\neg a \notin F$ and $b \notin G$ for any concordant filter $G \rset F$. Let $F = \{c \me d \mid a \prec c, \neg b \prec d\}$. By Lemma \ref{filt}, $F$ is a concordant filter if $c \me d \neq 0$ for any $a \prec c$, $\neg b \prec d$. But if $c\me d = 0$, then $ a \prec c \leq \neg d \prec \neg \neg b = b$, so $a \prec b$ by (A3), contradicting our assumption. Hence $F$ is a concordant filter. Now if $\neg a \in F$ there must be some $e \in F$ such that $e \prec \neg a$. But this means that $a \prec \neg e$ and therefore $\neg e \in F$, a contradiction. Similarly for any concordant $G \rset F$, if $b \in G$ there must be some $e \in G$ such that $e \prec b$. But then $\neg b \prec \neg e$ so $\neg e \in F \sset G$, a contradiction. Therefore $F \in \overline{\widehat{a}} \setminus \orclo \widehat{b}$.
\end{proof}

Putting Corollary \ref{cor1} and Lemma \ref{prox} together yields the desired representation theorem.

\begin{theorem} \label{obj1}
Let $V = (B, \prec)$ be a de Vries algebra with dual filter space $(S_V, \sigma)$. Then $V$ is isomorphic to $(\RO(S_V), \ll)$.
\end{theorem}

\section{De Vries Spaces}\label{space}
In this section, we characterize the choice-free duals of de Vries algebras. In other words, we give an axiomatization of topological spaces of the form $(S_V,\sigma)$ for some de Vries algebra $V$. In order to do so, we first need to introduce the following separation axioms:

\begin{definition}

\begin{enumerate}
\item[]
    \item  A topological space $(X, \tau)$ is \textit{order-regular} if for any closed set $B$ and any $x \notin \orint B$, there are disjoint open sets $U$, $V$ such that $x \in U$ and $\orint B \sset V$.
    \item  A topological space $(X, \tau)$ is \textit{order-normal} if for any closed set $A$ and any regular closed set $B$ such that $A$ is disjoint from $\orint B$, there are disjoint open sets $U$ and $V$ such that $A \sset \orclo U$ and $\orint B \sset V$.
\end{enumerate}
\end{definition}

Order-regularity and order-normality are straightforward variations of the usual regularity and normality separation axioms in general topology. Separation axioms for ordered topological spaces have been studied in the past \cite{mccartan1968separation,nachbin1965,Priestley}, but here we are concerned with a very specific kind of ordered topological spaces, in which the order is determined by the topology. In the case of compact $T_1$ spaces, these separation properties are essentially equivalent to Hausdorffness:

\begin{lemma}\label{sep}
Let $(X, \tau)$ be a compact $T_1$-space. The following are equivalent:
\begin{enumerate}
    \item $(X,\tau)$ is Hausdorff;
    \item $(X,\tau)$ is order-regular;
    \item $(X,\tau)$ is order-normal and order-regular.
\end{enumerate}
\end{lemma}

\begin{proof}
Recall that if $(X,\tau)$ is $T_1$, then the specialization preorder on $X$ is just the identity relation. Thus a $T_1$ order-regular space is regular Hausdorff, which implies that it is also Hausdorff. As compact Hausdorff spaces are also regular, this shows that (i) and (ii) are equivalent. Moreover, (iii) clearly implies (ii), and compact Hausdorff spaces are also normal, which for $T_1$ spaces implies order-normality, showing that (i) implies (iii).
\end{proof}

As we will now see, for spaces in which the regular opens are also order-regular open, order-normality suffices to establish that they form a de Vries algebras when equipped with the relation $\ll$ defined above.

\begin{lemma}
Let $(X, \tau)$ be an order-normal space such that $\RO(X) \sset \ROU(X)$. For any $U,V \in \RO(X)$, let $U \ll V$ iff $\overline{U} \sset \orclo V$. Then $(\RO(X),\ll)$ is a de Vries algebra.
\end{lemma}

\begin{proof}
Since $\RO(X)$ is a complete Boolean algebra, we only need to verify axioms (A1)-(A7):
\begin{description}
    \item[(A1)] $\bm{X \ll X}$. Immediate.
    \item[(A2)] $\bm{U \ll V}$ \textbf{implies} $\bm{U \sset V}$. Suppose $\overline{U} \sset \orclo V$. Taking complements, this yields $\mathord{-}\orclo V \sset U^\bot$. Because every closed set is a downset, $\orclo A \sset \overline{A}$ for any $A \sset X$, so $\orclo \mathord{-}\orclo V \sset \overline{U^\bot}$. Complementing again, we conclude that $U = U\reg \sset \orint \orclo V = V$.
    \item[(A3)] $\bm{U_1 \sset U_2 \ll V_1 \sset V_2}$ \textbf{implies} $\bm{U_1 \ll V_2}$. We have the following chain on inclusions: $\overline{U_1} \sset \overline{U_2} \sset \orclo V_1 \sset \orclo V_2$.
    \item[(A4)] $\bm{U \ll V_1}$ \textbf{and} $\bm{U\ll V_2}$ \textbf{together imply} $\bm{U\ll V_1 \cap V_2}$. Suppose both $\overline{U} \sset \orclo V_1$ and $\overline{U} \sset \orclo V_2$. Then since $U, V_1, V_2 \in \ROU(X)$, we have that $\mathord{-}\orclo (U^\bot) \sset V_1$ and $\mathord{-}\orclo (U^\bot) \sset V_2$, hence $\orclo \mathord{-}\orclo (U^\bot) \sset \orclo (V_1 \cap V_2)$. Now since $U^\bot \in \ROU(X)$, we have that $\orclo \mathord{-}\orclo (U^\bot) = \mathord{-}(U^\bot) = \overline{U}$, and therefore $\overline{U} \sset \orclo (V_1 \cap V_2)$.
    \item[(A5)] $\bm{U \ll V}$ \textbf{implies} $\bm{V^\bot \ll U^\bot}$. Suppose $\overline{U} \sset \orclo V$. Then $\orclo \mathord{-}\orclo V \sset \orclo (U^\bot)$. Taking complements, we have $-\orclo (U^\bot) \sset \orint\orclo V = V$ since $V \in \ROU(X)$. Now since $V \in \RO(X)$, $-V = \overline{V^\bot}$. Therefore, taking complements again, we have that $\overline{V^\bot} \sset \orclo (U^\bot)$, hence $V^\bot \ll U^\bot$.
    \item[(A6)] $\bm{U \ll V}$ \textbf{implies that there is} $\bm{W}$ \textbf{such that} $\bm{U \ll W \ll V}$. Suppose $\overline{U} \sset \orclo V$, and consider the set $X \setminus \orclo V = \orint \mathord{-}V$. As $\overline{U}$ and $\orint \mathord{-}V$ are disjoint and $-V$ is regular closed, by order-normality we get some disjoint open sets $W_1, W_2$ such that $\overline{U} \sset \orclo W_1$ and $\orint \mathord{-}V \sset W_2$. Note that this implies that $\overline{W_1} \cap \orint \mathord{-}V =\emptyset$, and therefore $\overline{W_1} \sset \orclo V$. Letting $W = W_1\reg$, we have that $\overline{U} \sset \orclo W_1 \sset \orclo W$, and $\overline{W} = \overline{W_1} \sset \orclo V$.
    \item[(A7)] \textbf{If} $\bm{U \neq \emptyset}$ \textbf{then there is} $\bm{V \neq \emptyset}$ \textbf{such that} $\bm{V \ll U}$. Suppose $U \neq \emptyset$ and let $x \in U$. Consider $X \setminus \orclo U = \orint \mathord{-}U$. Note that $\orclo x$ is disjoint from $\orint \mathord{-}U$ and is closed, since $\orclo x = \bigcap_{x \notin U, U\in \tau} \mathord{-}U$. By order-normality, we have disjoint open sets $V_1$ and $V_2$ such that $\orclo x \sset \orclo V_1$ and $\orint \mathord{-}U \sset V_2$. Note that this implies that $V_1 \neq \emptyset$ and that $\overline{V_1} \sset \orclo U$. Now letting $V = V_1\reg$, it follows that $V \neq \emptyset$ and $\overline{V} = \overline{V_1} \sset \orclo U$. 
\end{description}
Thus $(\RO(X),\ll)$ is a de Vries algebra.
\end{proof}

We are now in a position to define the choice-free duals of de Vries algebras:

\begin{definition}
A \textit{de Vries space} ($dV$-space for short) is a topological space $(X, \tau)$ satisfying the following conditions:
\begin{enumerate}
    \item $(X,\tau)$ is $T_0$, compact and order-normal;
    \item $\RO(X)$ is a basis for $\tau$ and $\RO(X) \sset \ROU(X)$;
    \item For every $x \in \RO(X)$, $\RO(x) = \{U \in \RO(X) \mid x \in U\}$ is a concordant filter on $\RO(X)$, and for every filter $F$ on $\RO(X)$, there is $x \in X$ such that $\concord F = \RO(x)$.
\end{enumerate}
\end{definition}

\begin{lemma} \label{lma3}
Let $V = (B, \prec)$ be a de Vries algebra. Then $(S_V, \sigma)$ is an order-regular $dV$-space.
\end{lemma}

\begin{proof}
Condition (ii) follows from Lemma \ref{basic}, and condition (iii) is immediate from Theorem \ref{obj1}, so we only have to check that $(S_V,\sigma)$ is $T_0$, compact, order-normal and order-regular. It is routine to verify that $(S_V, \sigma)$ is $T_0$. For compactness, note that $\concord\{1\} = \{1\} \in S_V$, so if $S_V \sset \bigcup_{a \in A} \widehat{a}$ for some $A \sset V$, it follows that $1 \in A$ and thus $S_V$ has a finite subcover. 

For order-regularity, let $B = \bigcap_{a \in A} -\widehat{a}$ be a closed set and $F \notin \orint B$. Then $F \in \orclo \mathord{-}B = \bigcup_{a \in A} \orclo \widehat{a}$, which means that there is $a \in A$ and $c \prec a$ such that $\neg c \notin F$. By $(A6)$ there is some $c' \in V$ such that $c \prec c' \prec a$. Now $F \in -\widehat{\neg c} = \overline{\widehat{c}} \sset \orclo \widehat{c'}$, and $-\widehat{\neg c'}=\overline{c'} \sset \orclo \widehat{a}$, so $\orint B \sset \widehat{\neg c'}$. Thus $\widehat{c'}$ and $\widehat{\neg c'}$ are the required open sets.

Finally, for order-normality, fix a closed set $U = \bigcap_{a \in A} -\widehat{a}$ and a regular closed set $B$ such that $\bigcap_{a \in A} -\widehat{a} \sset \orclo \mathord{-}B$. Because $B$ is regular closed it is of the form $-\widehat{b}$ for some $b \in V$. Now consider the concordant filter $F = \{c \in V \mid \neg b \prec c\}$. If there is $G \supseteq F$ such that $G \in \widehat{b}$, then there must be $c \in G$ such that $c \prec b$. But then $\neg c \in F \sset G$, and $G$ is not a proper filter, a contradiction. Thus $F \notin \orclo \widehat{b}$, which means that $F \in \bigcup_{a \in A} \widehat{a}$. Hence there is some $a \in A$ and some $c \in V$ such that $\neg b \prec c \prec a$, which in turn implies that $\neg a \prec \neg c \prec b$. This implies that $-\widehat{a} = \overline{\widehat{\neg a}} \sset \orclo \widehat{\neg c}$, and $-\widehat{c} = \overline{\widehat{\neg c}} \sset \orclo \widehat{b}$, and therefore we have two disjoint open sets, $\widehat{\neg c}$ and $\widehat{c}$, such that $\bigcap_{a \in A} \mathord{-}\widehat{a} \sset \orclo \widehat{\neg c}$ and $\orint \mathord{-}\widehat{b} \sset \widehat{c}$.
\end{proof}

\begin{theorem}\label{obj2}
Let $(X,\tau)$ be a $dV$-space. Then $(X,\tau)$ is homeomorphic to $(S_{(\RO(X),\ll)},\sigma)$.
\end{theorem}

\begin{proof}
Let $f : X \to S_{(\RO(X),\ll)}$ be given by $f(x) = \RO(x)$. Then $f$ is well-defined and surjective by condition (iii), and it is injective because $X$ is $T_0$. Moreover, for any $U \in \RO(X)$, we have that $x \in U$ iff $U \in \RO(x)$ iff $U \in f(x)$ iff $f(x) \in \widehat{U}$. By Theorem \ref{obj1} and since $\RO(X)$ is a basis for $X$, this is enough to conclude that $f$ is open and continuous and therefore a homeomorphism.
\end{proof}

Note that, as a corollary to Lemma \ref{lma3} and Theorem \ref{obj2}, we obtain that any $dV$-space is order-regular.

Let us conclude this section by characterizing $UV$-spaces as a special kind of $dV$-spaces. In order to do so, it is convenient to introduce first the following notion.

\begin{definition}
Let $(X,\tau)$ be a topological space. An open subset of $(X,\tau)$ is \textit{well rounded} if for any closed set $B$ such that $B \sset \orclo U$, there are disjoint open sets $V$ and $W$ such that $B \sset \orclo V$ and $-W \sset \orclo U$.
\end{definition}

Well-rounded subsets of a $dV$-space will play an important role later on when connecting our results with some standard notions of point-free topology. For now, let us note that a topological space in which every open is well-rounded is also order-regular and order-normal. Indeed,  order-normality amounts to the requirement that every regular open set be well-rounded, and order-regularity follows from the fact that $\orclo x$ is closed in every topological space. While not every open subset of a $dV$-space is well rounded, this is true for a special class of those, namely $UV$-spaces.

\begin{lemma}\label{uvchar}

A topological space $(X, \tau)$ is a $UV$-space if and only if it is a $dV$-space such that $(\RO(X), \ll)$ is \textit{zero-dimensional}.

\end{lemma}

\begin{proof} For the left-to-right direction, suppose $(X, \tau)$ is a $UV$-space. We may therefore view it as $(S_B, \sigma)$ for some  Boolean algebra $B$. This can be used to show that every open set $(X, \tau)$ is well-rounded.  Indeed, let $U = \bigcap_{a \in A} \mathord{-}\widehat{a}$ and $V = \bigcap_{c \in C} \mathord{-}\widehat{c}$ for some $A$, $C$, subsets of $B$ such that $\bigcap_{a \in A} \mathord{-}\widehat{a} \sset \orclo \bigcup_{c \in C} \widehat{c}$. Without loss of generality, we may assume that $C$ is a proper ideal: if $F \in \orclo{\widehat{c'}}$ for some $c' = c_1 \jo \dotsc \jo c_n$ with $c_1,...,c_n \in C$, then there must be some $i \leq n$ such that $\neg c_i \notin F$, and therefore $F \in \orclo \widehat{c_i}$. So let $F = \{ \neg c \mid c \in C\}$ be the dual filter of $C$. Clearly $F \notin \orclo \bigcup_{c \in C} \widehat{c}$, so $A \cap F \neq \emptyset$. This means that there is some $a \in A$ such that $\neg a \in C$. Thus $U \sset \mathord{-}\widehat{a} \sset \orclo \widehat{\neg a}$, and $\overline{\widehat{\neg a}} = \orclo \widehat{\neg a} \sset \orclo \mathord{-}V$. This shows that $(X, \tau)$ satisfies condition (i).

By \cite[Prop.~4.3.1]{hol19}, $\RO(X) \sset \ROU(X)$, so condition (ii) follows from condition (ii) of $UV$-spaces. Finally, condition (iii) follows from condition (iii) on $UV$-spaces once we show there is a one-to-one correspondence between concordant filters on $\RO(X)$ and proper filters on $B$, given by $F \mapsto \{a \in B \mid \widehat{a} \in F\}$. Recall first the observation that for any compact open set $U$ in a $UV$-space, $\overline{U} = \orclo U$ \cite[Prop.~4.1]{hol19}. This means that $\widehat{a} \ll \widehat{a}$ for any $W \in \mathsf{CO}\ROU(X)$. Now assume $U \ll V$ for some $U,V \in \RO(X)$. By \cite[Fact 8.2]{hol19}, we may write $U = \bigcup_{a \in A} \widehat{a}$ and $V = \bigcup_{c \in C} \widehat{c}$ for some ideals $A,C \sset B$. It is straightforward to see that $\overline{\bigcup_{a \in A} \widehat{a}} \sset \orclo \bigcup_{c \in C} \widehat{c}$ implies that there is $c \in C$ such that $a \leq c$ for all $a \in A$, and thus that $\overline{U} \sset \orclo \widehat{c}$ for some $c \in C$. Since $\widehat{c} \in \mathsf{CO}\ROU(X)$, we also have that $\overline{\widehat{c}} \sset \orclo \widehat{c} \sset \orclo V$, hence $U \ll \widehat{c} \ll \widehat{c} \ll V$. This shows that $\RO(X)$ is zero-dimensional. Moreover, if $F$ and $G$ are distinct concordant filters on $\RO(X)$, without loss of generality there is $V \in F \setminus G$. But then there is some $U \in F$ such that $U \ll V$, hence $U \ll \widehat{c} \ll V$ for some $c \in B$. This shows that the map $F \mapsto \{c \in B \mid \widehat{c} \in F\}$ is injective. For surjectivity, note that given any proper filter $G$ on $B$, $G' = \{ U \in \RO(X) \mid \exists c \in G: \widehat{c} \ll U\}$ will be a preimage of $G$. This completes the proof that $X$ is a $dV$-space such that $\RO(X,\ll)$ is a zero-dimensional de Vries algebra.

Conversely, suppose that $(X,\tau)$ is a $dV$-space such that $(\RO(X),\ll)$ is zero-dimensional. Let $B = \{ U \in \RO(X) \mid U \prec U\}$. Clearly, $B$ is a Boolean algebra, so we may consider its dual $UV$-space $UV(B)$. Since points in $X$ are in one-to-one correspondence with concordant filters on $\RO(X)$, by the same argument as above, there is a one-to-one correspondence between $X$ and $UV(B)$, given by $x \mapsto \{U \in B \mid x \in U\}$. As this map is easily seen to be a homeomorphism, it follows that $(X,\tau)$ is a $UV$-space.
\end{proof}

\section{Morphisms}\label{morph}

Having established the object part of our duality, the last step to obtain our duality result is to isolate the adequate notion of morphism between $dV$-spaces. It turns out to be a natural generalization of $UV$-maps:

\begin{definition}
Let $(X, \tau_1)$ and $(Y,\tau_2)$ be $dV$-spaces, and let $\leq_1$ and $\leq_2$ be the specialization orders induced by $\tau_1$ and $\tau_2$ respectively. A \textit{de Vries map} ($dV$-map for short) $f : X \to Y$ is a continuous function that is also weakly dense, i.e., is such that for any $x \in X$, if $f(x) \leq_2 y$ for some $y \in Y$, then there is $x' \geq_1 x$ such that $y \leq_2 f(x')$.
\end{definition}

Let $\cat{dVS}$ be the category of $dV$-spaces and $dV$-maps between them.
It is straightforward to verify that if $f : (X,\tau_1) \to (Y,\tau_2)$ is weakly dense, then for any upward-closed $V \sset Y$, $\orclo f\inv[V] = f\inv [\orclo V]$. This implies in particular that the preimage of any order-regular open set under a weakly dense map is order-regular open. This fact plays a role in the proof of the following lemma:

\begin{lemma} \label{mor1}
Let $f: (X,\tau_1) \to (Y,\tau_2)$ be a $dV$-map between $dV$-spaces. Then $\Phi(f) : (\RO(Y),\ll_2) \to (\RO(X),\ll_1)$, given by $\Phi(f)(U) = (f\inv[U])\reg$ for any $U \in \RO(Y)$, is a de Vries morphism.
\end{lemma}

\begin{proof}
We check the four conditions on de Vries morphisms in turn:

\begin{description}
    \item[(V1)] $\bm{\Phi(f)(\emptyset) = \emptyset}$. Immediate.
    \item[(V2)] $\bm{\Phi(f)(U \cap V) = \Phi(f)(U) \cap \Phi(f)(V)}$. Simply compute that: \begin{align*}
        \Phi(f)[U] \cap \Phi(f)[V] &= (f\inv[U])\reg \cap (f\inv[V])\reg \\&= (f\inv[U] \cap f\inv[V])\reg \\ &= \Phi(f)(U \cap V).
    \end{align*}
    \item[(V3)] $\bm{U \ll_2 V}$ \textbf{implies} $\bm{(\Phi(f)(U^\bot))^\bot \ll_1 \Phi(f)(V)}$. Suppose $\overline{U} \sset \orclo V$. This means that $f\inv[\overline{U}] \sset f\inv[\orclo V] = \orclo f\inv[V]$, since $f$ is weakly dense. Complementing, this gives us \[-\orclo f\inv[V] \sset f\inv[U^\bot] \sset \Phi(f)(U^\bot),\] which, using the fact that $f\inv[V]$ is order-regular open, yields \[-\orclo (\Phi(f)(U^\bot)) \sset f\inv[V] \sset \Phi(f)(V).\] Taking order-closure and complements again, this yields \[\mathord{-}\orclo (\Phi(f)(V)) \sset \Phi(f)(U^\bot) = (\Phi(f)(U^\bot))\reg ,\] and therefore \[\overline{(\Phi(f)(U^\bot))^\bot} \sset \orclo (\Phi(f)(V)).\]
    \item[(V4)] $\bm{\Phi(f)(V) = (\bigcup \{ \Phi(U) \mid U \ll_2 V\})\reg}$. The right-to-left direction is immediate. For the converse, suppose that $f(x) \in V$, and let $x' \geq_1 x$. Then $f(x') \in V$, which implies that there is some $U \ll_2 V$ such that $f(x') \in \orclo U$. This means that $f(x') \leq_2 y$ for some $y \in U$. Since $f$ is weakly-dense, there is $z \geq_1 x'$ such that $f(z) \geq_2 y$, and therefore $z \in \Phi(f)(U)$. This shows that $f\inv[V] \sset (\bigcup \{ \Phi(U) \mid U \ll_2 V\})\reg$, which clearly implies that $\Phi(f)(V) \sset (\bigcup \{ \Phi(U) \mid U \ll_2 V\})\reg$. 
\end{description}
Therefore $\Phi(f)$ is a de Vries morphism.
\end{proof}

It follows that we may define a contravariant functor $\Phi: \cat{dVS} \to \cat{deV}$ by letting $\Phi(X,\tau) = (\RO(X),\ll)$ for any $dV$-space $(X, \tau)$ and mapping any $f : (X,\tau_1) \to (Y,\tau_2)$ to $\Phi(f)$ as in Lemma \ref{mor1}. It is straightforward to verify that $\Phi$ preserves composition and identity arrows. Going from de Vries algebras to $dV$-spaces requires the following result:

\begin{lemma} \label{mor2}

Let $h : V_1 \to V_2$ be a de Vries morphism. Then the function $\Lambda(h) : (S_{V_2},\sigma_2) \to (S_{V_1},\sigma_1)$, given by $\Lambda(h)(F) = \concord h\inv[F]$ for any $F \in S_{V_2}$, is a $dV$-map.

\end{lemma}

\begin{proof}
Let us first show that $\Lambda(h)$ is continuous. For any $a \in V_1$ we compute:
\begin{align*}
    \Lambda(h)\inv[\widehat{a}] &= \{F \in S_{V_2} \mid \Lambda(h)(F) \in \widehat{a}\} \\ &= \{ F \in S_{V_2} \mid a \in \concord h\inv[F]\} \\ 
    &= \{F \in S_{V_2} \mid \exists c \prec a: h(c) \in F\} \\
    &= \bigcup_{c \prec a} \widehat{h(c).}
\end{align*}
Now we check that $\Lambda(h)$ is weakly dense. Let $F \in S_{V_2}$ and $G \in S_{V_1}$ be such that $\concord h\inv[F] \sset G$. I claim that \[H = \{a \in V_2\mid a \geq \neg h (\neg c) \me d \text{ for some }c \in G, d\in F\}\] is a concordant filter. To see that this is a proper subset of $V_2$, note that if $h(\neg c) \in F$ for some $c \in G$, then there is $c' \prec c \in G$, which implies that $\neg c \prec \neg c'$ and thus that $\neg c' \in G$, a contradiction. To see that $H$ is a filter, it is enough to verify that for any $c_1, c_2 \in G$, $\neg h (\neg c_1) \me \neg h (\neg c_2) \in H$. Since $G$ is concordant, there is $c' \in G$ such that $c' \prec c_1 \me c_2$, which implies that \[\neg h (\neg c) \prec \neg h (\neg(c_1 \me c_2)) \leq \neg h (\neg c_1) \me \neg h (\neg c_2),\] and therefore $\neg h (\neg c_1) \me \neg h (\neg c_2) \in H$. A similar argument shows that $\concord H = H$, which completes the proof of the claim.

By construction of $H$, $F \sset H$. Moreover, if $c \in G$, then there are $c_1, c_2 \in G$ such that $c_2 \prec c_1 \prec c$. Then $\neg h (\neg c_2) \prec h(c_1)$, which shows that $c \in \Lambda(h)[H]$, and therefore $G \sset \Lambda(h)[H]$. This completes the proof that $\Lambda(h)$ is a $dV$-map.
\end{proof}

We can therefore construct a functor $\Lambda : \cat{deV} \to \cat{dVS}$ by mapping any de Vries algebra $V$ to $\Lambda(V) = (S_V,\sigma)$ and any de Vries morphism $h$ to $\Lambda(h)$ as in Lemma \ref{mor2}. Again, it is straightforward to verify that $\Lambda$ preserves composition and identity arrows. We conclude with the main result of this paper:

\begin{theorem} \label{mainthm}
The functors $\Phi$ and $\Lambda$ establish a dual equivalence between the categories $\cat{deV}$ and $\cat{dVS}$.
\end{theorem}

\begin{proof}
In light of Theorems \ref{obj1} and \ref{obj2}, we only need to verify that:
\begin{enumerate}
    \item for any de Vries morphism $h : V_1\to V_2$, $\Phi\Lambda(h)(\widehat{a}) = \widehat{h(a)}$ for any $a \in V_1$;
    \item for any $dV$-map $f : (X,\tau_1) \to (Y, \tau_2)$, $\Lambda\Phi(f)(\RO(x)) = \RO(f(x))$ for any $x \in X$.
\end{enumerate}
For (i), it is enough to compute that:
\begin{align*}
    \Phi\Lambda(h)(\widehat{a}) &= ((\Lambda(h))\inv[\widehat{a}])\reg\\
    &=(\bigcup\{\widehat{h(b)} \mid b \prec a\})\reg\\
    &=\widehat{\bigjo\{h(b) \mid b \prec a\}}\\
    &=\widehat{h(a)}.
\end{align*}
For (ii), we first compute that:
\begin{align*}
    \Lambda\Phi(f)(\RO(x)) &= \concord (\Phi(f))\inv[\RO(x)]\\
    &=\concord\{U \mid \Phi(f)(U) \in \RO(x)\}\\
    &=\concord \{U \mid (f\inv[U])\reg \in \RO(x)\}.
\end{align*}
Now if $V \in \RO(f(x))$, then there is $U \ll V$ such that $U \in \RO(f(x))$, and therefore $x \in f\inv[U] \sset (f\inv [U])\reg$, and hence $V \in \Lambda\Phi(\RO(x))$. For the converse direction, suppose that $x \in (f\inv[U])\reg$ and that $U \ll V$ for some $U, V \in \RO(Y)$. I claim that for any $y \geq_2 f(x)$, $y \in \overline{U}$. Since $\overline{U} \sset \orclo V$, this implies that $f(x) \in \orint \orclo V$, and therefore that $v \in \RO(f(x))$. For the proof of the claim, note first that $x \in (f\inv[U])\reg$ implies that there is some regular open set $Z \in \RO(x)$ such that for any $x' \in Z$ and any open set $W$, $x' \in Z \cap W$ implies that $W \cap f\inv[U] \neq \emptyset$. Now fix some $y \in Y$ such that $f(x) \leq_2 y$. Since $f$ is weakly dense, there is $x' \geq_1 x$ such that $y \leq_2 f(x')$. The claim is proved if $f(x') \in \overline{U}$. Assume towards a contradiction that this is not the case. Then $x' \in f\inv[U^\bot]$, which is open since $f$ is continuous. But $x \leq_1 x'$ implies that $x' \in Z$, so $f\inv[U^\bot] \cap f\inv[U] \neq \emptyset$, a contradiction. This completes the proof.
\end{proof}

\section{Point-Free and Hyperspace Approaches}\label{hyp}
In this section, we relate $dV$-spaces to compact regular frames. Because both the equivalence between de Vries algebras and compact regular frames on the one hand, and the duality between de Vries algebras and $dV$-spaces on the other hand, do not rely on the Axiom of Choice, we already know that there is a choice-free duality between compact regular frames and $dV$-spaces. In order to describe this duality more precisely, we first need the following lemma:

\begin{lemma}\label{round}
For any de Vries algebra $V = (B, \prec)$, there is an order isomorphism between the poset $w\ORO(\Lambda(V))$ of well-rounded $\ORO$ subsets of $\Lambda(V)$ and the round ideals on $V$.
\end{lemma}

\begin{proof}
Let $\mathfrak{R}(V)$ be the frame of all round ideals of $V$ and $w\ORO(\Lambda(V))$ the poset of all well-rounded $\ORO$ subsets of $\Lambda(V)$ ordered by inclusion. Define $\alpha: \mathfrak{R}(V) \to w\ORO(\Lambda(V))$ as $I \mapsto \bigcup_{b \in I} \widehat{b}$ and $\beta: w\ORO(\Lambda(V)) \to \mathfrak{R}(V)$ as $ U \mapsto \{b \in B \mid \overline{\widehat{b}} \sset \orclo U\}$. I claim that $\alpha$ and $\beta$ are order preserving and inverses of one another.

First, let us verify that $\alpha(I)$ is a well-rounded $\ORO$ set for any round ideal $I$. Clearly, for any round ideal $I$, $\alpha(I)$ is open. To see that it is order-regular open, suppose $F \notin \alpha(I)$ for some concordant filter $F$, and consider the set $G = \{c \me \neg d\mid c \in F, d \in I\}$. I claim that $G \in \Lambda(V)$. Since $I$ is round, $I^\delta = \{\neg d \mid d \in I\}$ is a concordant filter, so by Lemma \ref{filt} we only need to verify that $c \me \neg d \neq 0$ for any $c \in F, d \in I$. But this follows immediately from the assumption that $F \notin \alpha(I)$. Thus $G \in \Lambda(V)$, and clearly we have that $F \sset G$ and $G \notin \orclo \alpha(I)$. Thus $F \notin \orint \orclo \alpha(I)$, which shows that $\alpha(I) \in \ROU(\Lambda(V))$. Finally, let us check that $\alpha(I)$ is well-rounded. Suppose  $W \sset \orclo \alpha(I)$ is a closed set of the form $\bigcap_{a \in A} \mathord{-}\widehat{a}$ for some $A \sset B$. Note that $I^\delta$ is a concordant filter and clearly $I^\delta \notin \orclo \alpha(I)$, so $A \cap I^\delta \neq \emptyset$. This means that $\neg a \in I$ for some $a \in A$. But then $\widehat{\neg a}$ and $\widehat{a}$ are the required open sets. This completes the proof that $\alpha(I) \in w\ORO(\Lambda(V))$.  

Conversely, let us show that for any $w\ORO$ set $U$, $\beta(U)$ is a round ideal. Clearly, $\beta(U)$ is downward closed. Now suppose we have $a,b \in V$ such that $\overline{\widehat{a}}$, $\overline{\widehat{b}} \sset \orclo U$. Then $\overline{\widehat{a} \cup \widehat{b}} = \overline{\widehat{a \jo b}}\sset \orclo U$. Since $U$ is well-rounded, there must be disjoint open sets $W_1,W_2$ such that $\overline{\widehat{a \cup b}} \sset \orclo W_1$ and $-W_2 \sset \orclo U$. By Theorem \ref{obj1}, $W_1\reg = \widehat{c}$ for some $c \in V$, and it is straightforward to verify that $\overline{\widehat{a \jo b}} \sset \orclo \widehat{c}$ and $\overline{\widehat{c}} \sset \orclo U$. This shows that $a \jo b \prec c$ and that $c \in \beta(U)$, establishing that $\beta(U)$ is a round ideal.

It is immediate to see that both maps are order preserving, so we only need to show that they are inverses of one another. Let $I$ be a round ideal. If $b \in I$, then $b \prec a$ for some $a \in I$. But then $\overline{\widehat{b}} \sset \orclo \widehat{a} \sset \orclo \alpha(I)$, so $b \in \beta\alpha(I)$. Conversely, assume $b \notin I$, and let $F = \{ c \me \neg d\mid b \prec c, d \in I\}$. If $c \me \neg d \leq \neg b$ for some $d \in I$ and $c$ such that $b \prec c$, then $b \me \neg d \prec c \me \neg d \leq \neg b$, hence $b \me \neg d \leq b \me \neg d \me \neg b \leq 0$. But this implies that $b \leq d$ and thus that $b \in I$, contradicting our assumption. Thus $\neg b \notin F$. By Lemma \ref{filt}, this shows that $F$ is a concordant filter and moreover $F \in \overline{\widehat{b}}$ by Lemma \ref{basic} (iv). But clearly $F \notin \orclo \alpha(I) = \bigcup_{d \in I} \orclo \widehat{d}$. By contraposition, it follows that if $\overline{\widehat{b}} \sset \orclo \alpha(I)$, then $b \in I$. This shows that $\beta\alpha(I) = I$ for any round ideal $I$.

Similarly, if $F \in U$ for $U \in w\ORO(\Lambda(V))$, then since $U$ is open there must be some $a \in F$ such that $\widehat{a} \sset U$. Since $F$ is concordant, there is $b \prec a$ for some $b \in F$. But then $F \in \widehat{b}$ and $\overline{\widehat{b}} \sset \orclo \widehat{a} \sset \orclo U$, so $F \in \alpha\beta(U)$. Conversely, suppose $F \in \alpha\beta(U)$. Then there is $a \in F$ such that $\overline{\widehat{a}} \sset \orclo U$. Since $\overline{\widehat{a}}=-\widehat{\neg a}$ and for any concordant $G\rset F$, $\neg a \notin G$, it follows that $F \in \orint \orclo U = U$. This shows that $\alpha\beta(U) = U$, which completes the proof. 
\end{proof}

As a consequence, the well-rounded $\ORO$ subsets of any $dV$-space form a compact regular frame, and we can lift this correspondence to a functor $w\ORO: \cat{deV} \to \cat{KRFrm}$. To go from compact regular frames to $dV$-spaces, it is enough to recall that the round ideals on a de Vries algebra $V$ are precisely the duals of concordant filters on $V$. Thus given a compact regular frame $L$, we may simply define the topological space $\Xi(L) = (L^-,\delta)$, where $L^- = L \setminus {1_L}$ and $\delta$ is the topology generated by sets of the form $\check{a} = \{b \mid \neg a \prec b\}$ for any $a \in L$. Indeed, since $L$ is isomorphic to $\mathfrak{R}(B(L))$, we may think of any $b \in L$ as a round ideal $I_b$ on the de Vries algebra $(B(L), \prec)$ such that for any $b \in L$ and $c \in B(L)$, $c \prec b$ iff $\neg c \in I_b$. But since $B(L)=\{\neg a \mid a \in L\}$, we therefore have for any $a \in L$:
\begin{align*}
    \check{a} &= \{b \in L^- \mid \neg a \prec b\} \\
    &= \{b \in L^- \mid \neg\neg a \in I_b\} \\
    &= \{b \in L^- \mid \neg a \in (I_b)^\delta\}\\
    &= \{b \in L^- \mid (I_b)^\delta \in \widehat{\neg a}\}.
\end{align*} 

\noindent This shows that the correspondence $b \mapsto (I_b)^\delta$ is a homeomorphism between $\Xi(L)$ and $\Lambda(B(L))$. It follows that $\Xi$ lifts to a contravariant functor from $\cat{KRFrm}$ to $\cat{dVS}$ and that we have the following theorem:

\begin{theorem}\label{chfis}
For any compact regular frame $L$, $L$ is isomorphic to $w\ORO(\Xi(L))$. Conversely, any $dV$-space $(X,\tau)$ is homeomorphic to $\Xi(w\ORO(X))$. Moreover, $w\ORO$ and $\Xi$ establish a duality between $\cat{KRFrm}$ and $\cat{dVS}$.
\end{theorem}

We may think of Theorem \ref{chfis} as establishing a choice-free analogue of Isbell duality. In the presence of (BPI), any compact regular frame is spatial, meaning that any compact regular frame $L$ is isomorphic to $\Omega(pt(L))$, or equivalently that any compact regular frame is the lattice of open sets of some compact Hausdorff space. In our choice-free case, we do not represent $L$ as the open sets of a topological space (since doing so would imply Isbell duality), but only as the well-rounded order-regular open sets of a $dV$-space. We might however be interested in better understanding the relationship between the Isbell dual of a compact regular frame and its de Vries dual. The answer turns out to involve the upper Vietoris functor on compact regular frames.

Recall that the Vietoris hyperspace of a compact Hausdorff space $(X,\tau)$ is obtained by taking as points the closed subsets of $X$. That a Vietoris-like construction would play a role in our duality is far from surprising. De Vries had already remarked \cite[Theorem I.3.12]{devries1962compact} that there was a dual order-isomorphism between the closed sets of a compact Hausdorff space and the concordant filters on its de Vries algebra of regular open sets. Moreover, assuming (BPI), the dual $UV$-space of a Boolean algebra $B$ is homeomorphic to the upper Vietoris hyperspace of the dual Stone space of $B$ \cite[Theorem 7.7]{hol19}. The upper Vietoris construction can also be defined on compact regular locales \cite{hol19,Johnstone,johnstone2020vietoris}:

\begin{definition}
Let $L$ be a compact regular locale. The \textit{upper Vietoris space} of $L$ is the topological space $UV(L) = (L^-,\tau_\Box)$, where $\tau_\Box$ is the topology generated by the sets $\Box a = \{ b \in L^- \mid a \jo b = 1_L\}$ for any $a \in L$.
\end{definition}

\begin{lemma}
For any locale $L$, $\Xi(L)$ is homeomorphic to $UV(L)$.
\end{lemma}

\begin{proof}
Since $\Xi(L)$ and $UV(L)$ have the same domain, it is enough to show that the two topologies coincide. For any $a \in L$:
\begin{align*}
    \check{a} &= \{b \in L^- \mid \neg a \prec b\}\\
    &= \{ b \in L^- \mid \neg\neg a \jo b = 1_L\}\\
    &= \Box \neg\neg a,
\end{align*}
which shows that $\delta \sset \tau_\Box$. Conversely, I claim that for any $a \in L$, \[\Box a = \bigcup_{b \prec a} \check{b}=\{ c \in L \mid \exists b \prec a :\neg b \prec c\}.\] To see this, notice first that if $\neg b \prec c$ for some $b \prec a$, then $c \jo \neg\neg b = 1_L$ and $\neg \neg b \leq a$, which implies that $a \jo c = 1_L$. This shows the right-to-left inclusion. For the converse, suppose that $a \jo c = 1_L$. Since $L$ is regular, $a = \bigjo \{b \in L \mid b \prec a\}$, and hence $1_L = \bigjo\{b \jo c \mid b \prec a\}$. Since $L$ is also compact, this means that there are $b_1,...,b_n$ such that $b_1\jo...\jo b_n \prec a$ and $c \jo b_1 \jo... \jo b_n = 1_L$.  Letting $b = \neg\neg(b_1\jo...\jo b_n)$, it follows that $b \prec a$ and that $\neg b \prec c$. This shows that $\Box a = \bigcup_{b \prec a} \check{b}$, and therefore that $\tau_\Box \sset \delta$.
\end{proof}

As an immediate corollary of the previous lemma, we obtain the following characterization of $dV$-spaces, which can be seen as a generalization of Theorem 7.7 in \cite{hol19}:

\begin{theorem}
A topological space is a $dV$-space if and only if it is homeomorphic to the upper Vietoris space of a compact regular locale.
\end{theorem}

Let us conclude this section by noting that connections between de Vries duality and the Vietoris functor on compact Hausdorff spaces have already been studied in \cite{bezhanishvili2015modalcompact,bezhanishvili2015modal}. In particular, the authors define modal de Vries algebras and prove that they are the duals of coalgebras of the Vietoris functor. For lack of space, we leave as an open problem the relationship between modal de Vries algebras and $dV$-spaces.

\section{Two Applications}\label{app}
We conclude by briefly mentioning two straightforward applications of the duality presented here. The first one is a choice-free version of Tychonoff's Theorem for compact Hausdorff spaces and the second one deals with the topological semantics of the strong implication calculus defined in \cite{hol19}.
\subsection{The Choice-free Product of Compact Hausdorff Spaces}
The following is a well-known result in point-free topology \cite{Johnstone,johnstone1981tychonoff,picado2011frames}:
\begin{lemma}
The category $\cat{KRFrm}$ is closed under coproducts.
\end{lemma}

By the duality obtained in the previous section, this means that the category of $dV$-spaces is closed under products. This means that a version of Tychonoff's Theorem for $dV$-spaces (the product in $\cat{dVS}$ of a family of $dV$-spaces is compact) holds in a choice-free setting. Moreover, this also motivates the following definition. 

\begin{definition}
Let $\{(X_i, \tau_i)\}_{i \in I}$ be a family of compact Hausdorff spaces. The \textit{choice free product} of this family is the $dV$-space $\Xi(\bigoplus_{i \in I} \Omega(X_i))$.
\end{definition}

As an immediate consequence of the results in the previous section, we get the following choice-free Tychonoff Theorem for compact Hausdorff spaces: 
\begin{theorem} \label{tych}
The choice-free product of a family of compact Hausdorff spaces $\{(X_i, \tau_i)\}_{i \in I}$ is compact. Moreover, under (BPI), it is homeomorphic to the upper-Vietoris space of $\prod_{i \in I} (X_i,\tau_i)$.
\end{theorem}

It is worth contrasting this result to one that can be obtained using Isbell duality. Since the category of compact regular frames is closed under coproducts, it can be proved without appealing to the Axiom of Choice that the coproduct of the frames of opens of any family $\{(X_i, \tau_i)\}_{i \in I}$ of compact Hausdorff spaces is a compact frame. Under (BPI), this frame is precisely the frame of opens of the product of $\{(X_i, \tau_i)\}_{i \in I}$ in the category of topological spaces. In the absence of (BPI) however, it may fail to be spatial. We may therefore see Theorem \ref{tych} as a \textit{semi-pointfree} version of Tychonoff's Theorem, that is choice-free yet remains spatial.

\subsection{Topological Completeness for the Symmetric Strong Implication Calculus}

De Vries duality has been used in \cite{bezhanishvili2019strict} to prove that the Symmetric Strong Implication Calculus $\mathsf{S^2IC}$ is sound and complete with respect to the class of compact Hausdorff spaces. This calculus is obtained by adding a binary relation $\rightsquigarrow$ to the language of classical propositional calculus, to be interpreted as a \textit{strong implication} connective. Given a contact algebra $(B, \prec)$, one can interpret the strong implication connective by letting $a \rightsquigarrow b = 1_B$ if $a \prec b$ and $a \rightsquigarrow b = 0$ otherwise. This gives rise to a binary normal and additive operator $\Delta(a,b) := \neg(a \rightsquigarrow \neg b)$, meaning that one may think of the pair $(B, \rightsquigarrow)$ as a BAO. For details on the axiomatization of $\mathsf{S^2IC}$, we refer to \cite{bezhanishvili2019strict}.
In order to provide a choice-free topological semantics for $\mathsf{S^2IC}$, we introduce the following notion:

\begin{definition}
A \textit{de Vries topological model} is a triple $(X,\tau,V)$ such that $(X, \tau)$ is a $dV$-space, and $V$ is a valuation such that for any formulas $\varphi$, $\psi$ of $\mathsf{S^2IC}$:
\begin{itemize}
    \item If $\varphi$ is propositional letter $p$, then $V(\varphi) \in \RO(X)$;
    \item $V(\neg \varphi) = V(\varphi)^\bot$ and $V(\varphi \me \psi) = V(\varphi) \cap V(\psi)$;
    \item $V(\varphi \rightsquigarrow \psi) = X$ if $\overline{V(\varphi)}\sset \orclo V(\psi)$ and $V(\varphi \rightsquigarrow \psi) = \emptyset$ otherwise.
\end{itemize}
A formula $\varphi$ is \textit{valid} on a $dV$-space $(X,\tau)$ iff $V(\varphi)=X$ for any de Vries topological model $(X,\tau,V)$.
\end{definition}

As a consequence of Theorem \ref{mainthm}, we have the following result, which does not assume the Axiom of Choice:

\begin{theorem}
The system $\mathsf{S^2IC}$ is sound and complete with respect to the class of all $dV$-spaces.
\end{theorem}

\begin{proof}
By Theorem 5.10 and Remark 5.11 in \cite{bezhanishvili2019strict}, de Vries algebras provide a sound and complete algebraic semantics for $\mathsf{S^2IC}$, and this result can be obtained choice-free. Combining this result with Theorem \ref{mainthm}, it follows that $dV$-spaces also provide a choice-free sound and complete semantics for $\mathsf{S^2IC}$.
\end{proof}

Since $dV$-spaces constitute a choice-free, filter-based representation of de Vries algebras, we may think of our choice-free de Vries duality as providing a possibility semantics for the logic of region-based theories of space, just as the choice-free Stone duality through $UV$-spaces serves as a foundation for possibility semantics for classical and modal propositional logic \cite{holposs,Hol16,holliday2019complete}.

\section*{Acknowledgments}
I thank Wes Holliday and an anonymous referee for helpful comments that greatly improved the clarity of the paper.

\bibliographystyle{aiml22}
\bibliography{aiml22}

\end{document}